\documentclass[12pt]{article}
\usepackage{amsthm,amssymb}

\usepackage{nccbbb}
\usepackage{bm}
\def\sign{\mbox{\rm sign\,}}

{\theoremstyle{plain}%
  \newtheorem{theorem}{Theorem}
  \newtheorem{corollary}{Corollary}
  \newtheorem{proposition}{Proposition}
  \newtheorem{lemma}{Lemma}%
}
{\theoremstyle{remark}

\newtheorem*{remark}{Remark}
}
{\theoremstyle{definition}

}
\newtheorem*{thmA}{Theorem A}
\newtheorem*{thmB}{Theorem B}
\newtheorem*{thmC}{Theorem C}

\begin{document}
\renewcommand{\theequation}{\arabic{equation}}
\title{Sum formulas of multiple zeta values with arguments are multiples of a positive integer}
\author{Kwang-Wu Chen\footnote{
Department of Mathematics, University of Taipei,
No. $1$,  Ai-Guo West Road, Taipei $10048$, Taiwan.
E-mail: kwchen@uTaipei.edu.tw},
Chan-Liang Chung\footnote{
Institute of Mathematics, Academia Sinica, $6$F, 
Astronomy-Mathematics Building, 
No. $1$, Sec. $4$, Roosevelt Road, 
Taipei $10617$, Taiwan.
E-mail: andrechung@gate.sinica.edu.tw},
Minking Eie\footnote{
Department of Mathematics, National Chung Cheng University,
$168$ University Rd., Minhsiung, Chiayi $62102$, Taiwan.
E-mail: minking@math.ccu.edu.tw}
}
\maketitle

\begin{abstract}  
For $k\leq n$, let $E(mn,k)$ be the sum of all multiple zeta values of depth $k$ and weight $mn$
with arguments are multiples of $m\geq 2$. More precisely, 
$E(mn,k)=\sum_{|\bm\alpha|=n}\zeta(m\alpha_1,m\alpha_2,\ldots, m\alpha_k)$. 
In this paper, we develop a formula to express $E(mn,k)$ in terms of 
$\zeta(\{m\}^p)$ and $\zeta^\star(\{m\}^q)$, $0\leq p,q\leq n$. In particular,
we settle Gen\v{c}ev's conjecture on the evaluation of $E(4n,k)$ and 
also evaluate $E(mn,k)$ explicitly for  small even $m\leq 8$.
\end{abstract}

\noindent{\small {\it Key Words:} 
Multiple Zeta Values, Multiple Zeta-Star Values.

}

\setlength{\baselineskip}{18pt}
\section{Introduction}\label{sec.1}
The multiple zeta values (MZVs) and the multiple zeta star values (MZSVs) are defined by 
\cite{BBBL2, Eie, Eie2, Gen}
\begin{eqnarray*}
\zeta(\alpha_1,\alpha_2,\ldots,\alpha_r) &=& \sum_{1\leq k_1<k_2<\cdots<k_r}
k_1^{-\alpha_1}k_2^{-\alpha_2}\cdots k_r^{-\alpha_r}\quad\mbox{and}\\
\zeta^\star(\alpha_1,\alpha_2,\ldots,\alpha_r) &=&
\sum_{1\leq k_1\leq k_2\leq\cdots\leq k_r}
k_1^{-\alpha_1}k_2^{-\alpha_2}\cdots k_r^{-\alpha_r}
\end{eqnarray*}
with positive integers $\alpha_1,\alpha_2,\ldots,\alpha_r$ and $\alpha_r\geq 2$ for the sake
of convergence. The numbers $r$ and $|\bm\alpha|=\alpha_1+\alpha_2+\cdots+\alpha_r$
are the depth and weight of $\zeta(\bm\alpha)$ and $\zeta^\star(\bm\alpha)$. For our convenience,
we let $\{a\}^k$ be $k$ repetitions of $a$ such that 
$\zeta(\{2\}^3)=\zeta(2,2,2)$ and $\zeta^\star(\{3\}^4,5)=\zeta^\star(3,3,3,3,5)$.

MZVs and MZSVs are strongly connected with each other, for example,
\begin{eqnarray*}
\zeta^\star(s_1,s_2) &=& \zeta(s_1,s_2)+\zeta(s_1+s_2)\quad\mbox{and}\\
\zeta^\star(s_1,s_2,s_3) &=& \zeta(s_1,s_2,s_3)+\zeta(s_1+s_2,s_3)+\zeta(s_1,s_2+s_3)+\zeta(s_1+s_2+s_3).
\end{eqnarray*}

A principal goal in the theoretical study of MZVs or MZSVs 
is to determine all possible algebraic relations among them.
Several explicit values are interesting and known for special index sets
(e.g. \cite{BBBL1, BBBL2, Mun, Zag}).
For example, Zagier \cite{Zag} evaluated the 
value of $\zeta(\{2\}^a,3,\{2\}^b)$.
For a partition ${\cal P}=\{P_1,P_2,\ldots,P_\ell\}$ of the set $[n]=\{1,2,\ldots,n\}$,
let $c(\cal P)=($card$\,P_1-1)!($card$\,P_2-1)!\cdots($card$\,P_\ell-1)!$.
Given a $n$-tuple $\mathbf i=\{i_1,i_2,\ldots,i_n\}$, we define
$$
\zeta(\mathbf i,{\cal P})=\prod^\ell_{s=1}\zeta\left(\sum_{j\in P_s}i_j\right).
$$
Let $S_n$ be the symmetric group of degree $n$.
In 1992, Hoffman \cite{Hof} gave the evaluations of 
$\sum_{\sigma\in S_n}\zeta(i_{\sigma(1)},i_{\sigma(2)},
\ldots,i_{\sigma(n)})$ and $\sum_{\sigma\in S_n}\zeta^\star(i_{\sigma(1)},i_{\sigma(2)},
\ldots,i_{\sigma(n)})$.
\begin{proposition}[\cite{Hof}, Theorem 2.1, 2.2] \label{pro.01} 
For any real $i_1,i_2,\ldots,i_n>1$, we have
\begin{eqnarray} \label{eq.01}
\sum_{\sigma\in S_n}\zeta(i_{\sigma(1)},i_{\sigma(2)},
\ldots,i_{\sigma(n)}) &=&\sum_{\forall{\cal P}=\{P_1,\ldots,P_\ell\}\mbox{ of }[n]}
(-1)^{n-\ell}c({\cal P})\zeta(\mathbf i,{\cal P}),\\ \label{eq.02} 
\sum_{\sigma\in S_n}\zeta^\star(i_{\sigma(1)},i_{\sigma(2)},
\ldots,i_{\sigma(n)}) &=&\sum_{\forall{\cal P}=\{P_1,\ldots,P_\ell\}\mbox{ of }[n]}
c({\cal P})\zeta(\mathbf i,{\cal P}).
\end{eqnarray}
\end{proposition}

It is known that for $m\geq 2$, the infinite product
\begin{eqnarray*}
\prod^\infty_{n=1}\left(1+\frac{x^m}{n^m}\right) &=& 
1+\sum^\infty_{k=1}\frac 1{k^m}x^m
+\sum_{1\leq k_1<k_2}\frac 1{k_1^mk_2^m}x^{2m}+\cdots\\
&&+\sum_{1\leq k_1<k_2<\cdots<k_n}\frac 1{k_1^mk_2^m\cdots k_n^m}x^{mn}+\cdots.
\end{eqnarray*}
is the generating function of the sequence
$$
1,\zeta(m),\zeta(m,m),\ldots,\zeta(\{m\}^n),\ldots.
$$

On the other hand, the generating function of $\zeta^\star(\{m\}^n)$ 
\cite{Hof2, Mun} is given by
$$
\prod^\infty_{n=1}\left(1-\frac{x^m}{n^m}\right)^{-1}
=\prod^\infty_{n=1}\left(1+\frac{x^m}{n^m}+\frac{x^{2m}}{n^{2m}}
+\cdots+\frac{x^{km}}{n^{km}}+\cdots\right),
$$
so that $\zeta^\star(\{m\}^n)$ has another expression
$$
\sum^n_{k=1}\sum_{|\bm\alpha|=n}\zeta(m\alpha_1,m\alpha_2,\ldots,m\alpha_k).
$$
Let $E(mn,k)$ be the sum of all multiple zeta values of depth $k$ and weight $mn$
with arguments are multiples of $m\geq 2$. More precisely,
$$
E(mn,k)=\sum_{|\bm\alpha|=n}\zeta(m\alpha_1,m\alpha_2,\ldots,m\alpha_k).
$$
This sum of multiple zeta values $E(mn,k)$
is just a  part of $\zeta^\star(\{m\}^n)$.

Gangl, Kaneko, and Zagier \cite{GKZ} proved that $E(2n,2)=\frac34\zeta(2n)$, 
for $n\geq 2$. Shen and Cai \cite{SC} gave formulas for $E(2n,3)$ and $E(2n,4)$
in terms of $\zeta(2n)$ and $\zeta(2)\zeta(2n-2)$. Using an explicit generating function
for $E(2n,k)$, Hoffman \cite{Hof2} gave a general formula of $E(2n,k)$ for $k\leq n$.

Here we are able to express $E(mn,k)$ in terms of $\zeta(\{m\}^p)$ and 
$\zeta^\star(\{m\}^q)$ with $p+q=n$ and $p\geq k$. By convension we 
let $\zeta(\{m\}^0)=1$ and $\zeta^\star(\{m\}^0)=1$.
\begin{thmA} \label{thm.A}  
For a pair of positive integers $m$, $n$ with $m\geq 2$, we have for $1\leq k\leq n$,
$$
E(mn,k)=\sum_{p+q=n}(-1)^{p-k}{p\choose k}\zeta(\{m\}^p)\zeta^\star(\{m\}^q).
$$
\end{thmA}
By inputing the explicit values of $\zeta(\{m\}^p)$ and $\zeta^\star(\{m\}^q)$
for various $m$, we obtain the sum $E(mn,k)$. In particular, the following are most interesting 
so we list them as theorems.
\begin{thmB} \label{thm.B} 
For a pair of positive integers $n$, $k$ with $1\leq k\leq n$, we have
\begin{eqnarray*}
E(2n,k) &=& \sum_{|\bm\alpha|=n}\zeta(2\alpha_1,2\alpha_2,\ldots,2\alpha_k)\\
&=& \frac{(-1)^{n-k}\pi^{2n}}{(2n+1)!}\sum^{n-k}_{q=0}
{n-q\choose k}{2n+1\choose 2q}2^{2q}B_{2q}\left(\frac 12\right).
\end{eqnarray*}
\end{thmB}

From the general identity \cite{Eie, Eie2, Rad}
$$
B_n(kx)=k^{n-1}\sum^{k-1}_{j=0}B_n\left(x+\frac jk\right),
$$
we have
$$
2^{2q}B_{2q}\left(\frac 12\right)=(2-2^{2q})B_{2q}.
$$
So our Theorem B coincides with the result of Hoffman \cite{Hof2}
$$
E(2n,k)=\frac{(-1)^{n-k-1}\pi^{2n}}{(2n+1)!}\sum^{n-k}_{j=0}
{n-j\choose k}{2n+1\choose 2j}2(2^{2j-1}-1)B_{2j}.
$$
The evaluation of $E(4n,k)$ appeared in Gen\v{c}ev's paper \cite{Gen} as a conjecture.
Here we confirm the conjecture is true.
\begin{thmC} \label{thm.C} 
For a pair of positive integers $n$, $k$ with $1\leq k\leq n$, we have
\begin{eqnarray*}
\lefteqn{E(4n,k)} \\
&=& \frac{2^{2n+1}\pi^{4n}}{(4n+2)!}\sum^{n-k}_{q=0}(-1)^{n-q-k}
{n-q\choose k}{4n+2\choose 4q}\sum_{|\bm q|=2q}
{4q\choose 2q_1}(-1)^{q_1}2^{2q}B_{2q_1}\left(\frac 12\right)B_{2q_2}\left(\frac 12\right).
\end{eqnarray*}
\end{thmC}
Our purpose of this paper is to enable the evaluation of $E(2mn,k)$ directly by the fact 
that the evaluations of $\zeta(\{2m\}^p)$ and $\zeta^{\star}(\{2m\}^q)$ are known. 
This paper is organized as follows. In Section 2, we present proofs of Theorem A and Theorem B. 
In Section 3 and 4, we rewrite several evaluations of $\zeta(\{2m\}^n)$ and $\zeta^{\star}(\{2m\}^n)$ 
in the literature by dividing into two cases according to $m$ is even or odd. 
In addition, we give a different expression of $E(2mn,k)$ in the final section.
%
%
\section{Proofs of Theorem A and Theorem B}\label{sec.2}
In order to evaluate the special values at even integers of the Riemann zeta function
$$
\zeta(s)=\sum^\infty_{n=1}\frac 1{n^s},\quad\Re(s)>1,
$$
Euler develped the infinite product formula of the sine function
$$
\sin\pi x=\pi x\prod^\infty_{n=1}\left(1-\frac{x^2}{n^2}\right).
$$
The infinite product
$$
\prod^\infty_{n=1}\left(1-\frac{x^2}{n^2}\right)
$$
is the generating function of $(-1)^n\zeta(\{2\}^n)$.
Indeed we have
\begin{eqnarray*}
\prod^\infty_{n=1}\left(1-\frac{x^2}{n^2}\right) &=&
1-\left(\sum^\infty_{k=1}\frac1{k^2}\right)x^2+
\left(\sum_{1\leq k_1<k_2}\frac1{k_1^2k_2^2}\right)x^4+\cdots \\
&&\qquad\qquad+(-1)^n\left(\sum_{1\leq k_1<k_2<\cdots<k_n}\frac1{k_1^2k_2^2\cdots k_n^2}\right)x^{2n}
+\cdots \\
&=& 1-\zeta(2)x^2+\zeta(2,2)x^4+\cdots+(-1)^n\zeta(\{2\}^n)x^{2n}+\cdots.
\end{eqnarray*}
Using the power series expansion
$$
\frac{\sin \pi x}{\pi x}=1-\frac{(\pi x)^2}{3!}+\frac{(\pi x)^4}{5!}
+\cdots+\frac{(-1)^n(\pi x)^{2n}}{(2n+1)!}+\cdots,
$$
then it implies that the evaluation
$$
\zeta(\{2\}^n)=\frac{\pi^{2n}}{(2n+1)!}.
$$
On the other hand, the inverse of the infinite product
$$
\prod^\infty_{n=1}\left(1-\frac{x^2}{n^2}\right)^{-1}
=\prod^\infty_{n=1}\left\{1+\frac{x^2}{n^2}+\left(\frac{x^2}{n^2}\right)^2+\cdots
+\left(\frac{x^2}{n^2}\right)^k+\cdots\right\}
$$
is the generating function of the sequence of sum of multiple zeta values \cite{Hof2}
$$
\zeta^\star(\{2\}^n)=\sum^n_{k=1}E(2n,k)
=\sum^n_{k=1}\sum_{|\bm\alpha|=n}\zeta(2\alpha_1,2\alpha_2,\ldots,2\alpha_k).
$$
Also such an infinite product, according to the infinite product formula of the sine function,
is equal to 
$$
\frac{\pi x}{\sin \pi x}
$$
or
$$
\frac{2\pi i xe^{\pi i x}}{e^{2\pi i x}-1}
$$
and has the power series expansion 
$$
\sum^\infty_{n=0}\frac{(2\pi i x)^{2n}}{(2n)!}B_{2n}\left(\frac 12\right),
$$
where $B_n(x)$ are Bernoulli polynomials defined by 
$$
\frac{te^{xt}}{e^t-1}=\sum^\infty_{n=0}\frac{B_n(x)t^n}{n!} \quad \mbox{for} \; n=0,1,2,\ldots.
$$
Here we summarize our previous discussion as follows which we needed
in the evaluation of $E(2n,k)$.
\begin{proposition}[\cite{Hof2}] \label{prop.2} 
For any positive integer $n$, we have
\begin{equation}\label{eq.03} 
\zeta(\{2\}^n)=\frac{\pi^{2n}}{(2n+1)!},
\end{equation}
\begin{equation} \label{eq.04} 
\zeta^\star(\{2\}^n)=\sum^n_{k=1}\sum_{|\bm\alpha|=n}\zeta(2\alpha_1,2\alpha_2,\ldots,2\alpha_k)
=(2\pi i)^{2n}\frac{B_{2n}\left(\frac 12\right)}{(2n)!}
=\frac{(-1)^n\pi^{2n}}{(2n)!}(2-2^{2n})B_{2n}.
\end{equation}
\end{proposition}
Now we are ready to prove Theorem A and Theorem B.
\begin{proof}[Proof of Theorem A]
For real number $\lambda$, we consider the infinite product
$$
F_m(\lambda,x)=\prod^\infty_{n=1}\left(1+\frac{\lambda x^m}{n^m}\right)
\left(1-\frac{x^m}{n^m}\right)^{-1}
$$
which is a product of two infinite products
$$
\prod^\infty_{n=1}\left(1+\frac{\lambda x^m}{n^m}\right)\quad\mbox{and}\quad
\prod^\infty_{n=1}\left(1-\frac{x^m}{n^m}\right)^{-1}.
$$
The above products are generating functions of
$$
\lambda^n\zeta(\{m\}^n)\quad\mbox{and}\quad
\zeta^\star(\{m\}^n),
$$
respectively. Therefore, $F_m(\lambda,x)$ is the generating function of the convolution of
the two corresponding sequences. Hence the coefficient of $x^{mn}$ of $F_m(\lambda,x)$
is given by
$$
\sum_{p+q=n}\lambda^p\zeta(\{m\}^p) \zeta^\star(\{m\}^q).
$$
On the other hand, we rewrite $F_m(\lambda,x)$ as 
$$
F_m(\lambda,x)=\prod^\infty_{n=1}\left\{
1+(\lambda+1)\frac{x^m}{n^m}+(\lambda+1)\frac{x^{2m}}{n^{2m}}+\cdots+
(\lambda+1)\frac{x^{km}}{n^{km}}+\cdots\right\}
$$
and its coefficient of $x^{mn}$ is given by 
$$
\sum^n_{r=1}(\lambda+1)^rE(mn,r).
$$
This leads to the identity
\begin{eqnarray}\label{eq.05} 
\sum^n_{r=1}(\lambda+1)^rE(mn,r)=\sum_{p+q=n}\lambda^p\zeta(\{m\}^p)
\zeta^\star(\{m\}^q).
\end{eqnarray}
Differentiate both sides of the above identity with respect to $\lambda$ for $k$ times
and then set $\lambda=-1$, we obtain that 
$$
E(mn,k)=\sum_{p+q=n}(-1)^{p-k}{p \choose k}\zeta(\{m\}^p)\zeta^\star(\{m\}^q).
$$
\end{proof}

Taking $\lambda=1$ into Eq.(\ref{eq.05}), we have the following.
\begin{corollary} 
For a pair of positive integers $m, n$ with $m\geq 2$, we have for $1\leq k\leq n$,
$$
\sum_{r=1}^n 2^r E(mn,r)=\sum_{p+q=n}\zeta(\{m\}^p)\zeta^{\star}(\{m\}^q).
$$
\end{corollary}

\begin{proof}[Proof of Theorem B]
Substitute $\zeta(\{2\}^p)$ and $\zeta^\star(\{2\}^q)$ by Eq.(\ref{eq.03}) and Eq.(\ref{eq.04}) into Theorem A, we obtain that 
\begin{eqnarray*}
E(2n,k) &=& \sum_{p+q=n}(-1)^{p-k}{p\choose k}\frac{\pi^{2p}}{(2p+1)!}
(2\pi i)^{2q}\frac{B_{2q}\left(\frac12\right)}{(2q)!}\\
&=&\frac{(-1)^{n-k}\pi^{2n}}{(2n+1)!}\sum^{n-k}_{q=0}
{n-q\choose k}{2n+1\choose 2q}2^{2q}B_{2q}\left(\frac 12\right).
\end{eqnarray*}
\end{proof}
%
%
\section{Evaluations of $\zeta(\{2m\}^n)$} \label{sec.03}
The evaluation of $\zeta(\{2m\}^n)$ is available in \cite{AK,Mun}.
However, what we need is the explicit value of $\zeta(\{2m\}^n)$.
So we calculate them directly here for the reason of self-content.
\begin{theorem} \label{thm.01} 
Suppose that $m$ and $n$ are positive integers with $m\geq 3$ odd. Let $w=e^{2\pi i/m}$, then we have
$$
\zeta(\{2m\}^n)=\frac{2\cdot(2\pi)^{2mn}}{(2mn+m)!}
\left\{
\sum^{(m-1)/2}_{r=1}(-1)^{r-1}\sum_{0\leq j_1<j_2<\cdots<j_r\leq m-1}
(w^{j_1}+w^{j_2}+\cdots+w^{j_r})^{2mn+m}
\right\}.
$$
\end{theorem}
\begin{proof}
The infinite product formula of the sine function implies that for $|x|<1$ and $x\neq 0$,
\begin{eqnarray*}
1+\sum^\infty_{n=1}(-1)^n\zeta(\{2m\}^n)x^{2mn} &=&
\prod^\infty_{n=1}\left(1-\frac{x^{2m}}{n^{2m}}\right) \\
&=& \prod^\infty_{n=1}\prod^{m-1}_{j=0}\left(1-\frac{(w^jx)^2}{n^2}\right) 
\ =\  \frac 1{(\pi x)^m}\prod^{m-1}_{j=0}\sin(w^j\pi x).
\end{eqnarray*}
Now we express the product of sine functions into a linear combinations of sine functions as
$$
\prod^{m-1}_{j=0}\sin(w^j\pi x)
=\frac{(-1)^{(m-1)/2}}{2^{m-1}}\sum_{\varepsilon_j=\pm 1\atop{1\leq j\leq m-1}}
\sign(\varepsilon)\cdot \sin( \pi x(w^{m-1}+\varepsilon_1w^{m-2}+\cdots+\varepsilon_{m-1}))
$$
with
$$
\sign(\varepsilon)=\left\{\begin{array}{ll}
1,&\mbox{if the number of $-1$ in 
$\varepsilon=(\varepsilon_1,\varepsilon_2,\ldots,\varepsilon_{m-1})$ is even;}\\
-1, & \mbox{otherwise.}\end{array}\right.
$$
As $w^m=1$ and $m\geq 3$, 
$$
w^{m-1}+w^{m-2}+\cdots+w+1=0.
$$
So we are able to express the product as 
$$
\frac{(-1)^{(m-1)/2}}{2^{m-1}}\sum^{(m-1)/2}_{r=1}
(-1)^{r-1}\sum_{0\leq j_1<j_2<\cdots<j_r\leq m-1}
\sin(2\pi x(w^{j_1}+w^{j_2}+\cdots+w^{j_r})).
$$
The coefficient of $x^{2mn+m}$ of the above combination of sine function
gives the evaluation of $\zeta(\{2m\}^n)$.
\end{proof}
Here are some evaluations of $\zeta(\{2m\}^n)$:
\begin{eqnarray}\label{eq.06} 
\zeta(\{6\}^n) &=& \frac{6\cdot (2\pi)^{6n}}{(6n+3)!},\\ \label{eq.07} 
\zeta(\{10\}^n) &=& \frac{10\cdot(2\pi)^{10n}}{(10n+5)!}\left\{
1-\left(2\cos\frac{2\pi}{5}\right)^{10n+5}-\left(2\cos\frac{4\pi}{5}\right)^{10n+5}\right\},\\
\zeta(\{14\}^n) &=&\frac{14\cdot(2\pi)^{14n}}{(14n+7)!} \label{eq.08} 
\left\{1-\sum_{j=1}^3\left(2\cos\frac{2j\pi}{7}\right)^{14n+7}
+\sum^3_{j=1}\left(1+2\cos\frac{2j\pi}{7}\right)^{14n+7}\right.\\ \nonumber
&&\qquad\left.+\left(\frac{-1+i\sqrt{7}}{2}\right)^{14n+7}
+\left(\frac{-1-i\sqrt{7}}{2}\right)^{14n+7}
\right\}.
\end{eqnarray}
\begin{theorem}\label{thm.02} 
Suppose that $m$ and $n$ are positive integers with $m$ even. Then
$$
\zeta(\{2m\}^n)=\frac{(-1)^{n+1+\frac m2}\pi^{2mn}}{2^{m-2}(2mn+m)!}\cdot
\Im\left(\sum_{{\footnotesize\sign}(\varepsilon)=1\atop{\varepsilon_j=\pm 1,1\leq j\leq m-1}}
(w^{m-1}+\varepsilon_1w^{m-2}+\cdots+\varepsilon_{m-1})^{2mn+m}\right)
$$
with $w=\exp(2\pi i/2m)$ and 
$$
\sign(\varepsilon)=\left\{\begin{array}{ll}
1,&\mbox{if the number of $-1$ in 
$\varepsilon=(\varepsilon_1,\varepsilon_2,\ldots,\varepsilon_{m-1})$ is even;}\\
-1, & \mbox{otherwise.}\end{array}\right.
$$
\end{theorem}
\begin{proof}
The infinite product formula of the sine function implies that for $|x|<1$ and $x\neq 0$,
\begin{eqnarray}\nonumber
1+\sum_{n=1}^\infty(-1)^n\zeta(\{2m\}^n)x^{2mn}
&=&\prod^\infty_{n=1}\left(1-\frac{x^{2m}}{n^{2m}}\right)
\ =\ \prod^\infty_{n=1}\prod^{m-1}_{j=0}\left(1-\frac{(w^jx)^2}{n^2}\right)\\
&=&\frac{w^{-m(m-1)/2}}{\pi^mx^m}\prod^{m-1}_{j=0}\sin(w^j\pi x).
\label{eq.09} 
\end{eqnarray}
Now the product of sine function can be expressed as
\begin{eqnarray*}
\lefteqn{\frac{(-1)^{m/2}}{2^{m-1}}\sum_{\varepsilon_j=\pm 1\atop{1\leq j\leq m-1}}
\sign(\varepsilon)\cdot\cos(\pi x(w^{m-1}+\varepsilon_1 w^{m-2}+\cdots+\varepsilon_{m-1}))}\\
&=&\frac{(-1)^{m/2}}{2^{m-1}}\sum_{\varepsilon_j=\pm 1\atop{1\leq j\leq m-1}}
\sign(\varepsilon)\cdot\sum^\infty_{n=0}\frac{(-1)^n}{(2n)!}
\left[\pi x(w^{m-1}+\varepsilon_1 w^{m-2}+\cdots+\varepsilon_{m-1})
\right]^{2n}.
\end{eqnarray*}
The coefficient of $x^{2mn+m}$ of Eq.\,(\ref{eq.09}) then gives the evaluation of 
$\zeta(\{2m\}^n)$ up to a constant.
\begin{equation}\label{eq.10} 
\zeta(\{2m\}^n)=\frac{(-1)^nw^{-m(m-1)/2}\pi^{2mn}}{2^{m-1}(2mn+m)!}
\sum_{\varepsilon_j=\pm 1\atop{1\leq j\leq m-1}}
\sign(\varepsilon)\cdot (w^{m-1}+\varepsilon_1w^{m-2}+\cdots+\varepsilon_{m-1})^{2mn+m}.
\end{equation}
Let
\begin{eqnarray*}
A &=& \{w^{m-1}+\varepsilon_1w^{m-2}+\cdots+\varepsilon_{m-1}\,:\,
\sign(\varepsilon)=1\}, \\
B &=& \{w^{m-1}+\varepsilon_1w^{m-2}+\cdots+\varepsilon_{m-1}\,:\,
\sign(\varepsilon)=-1\}.
\end{eqnarray*}
Choose $B_j\in B$ and write $B_j=w^{m-1}+\varepsilon_1w^{m-2}+\cdots+\varepsilon_{m-1}$.
If $\varepsilon_{m-2}=1$, then
\begin{eqnarray*}
-\overline{B_j} &=& w^m(w^{-m+1}+\varepsilon_1w^{-m+2}+\cdots+
\varepsilon_{m-3}w^{-2}+w^{-1}+\varepsilon_{m-1}) \\
&=& w^{m-1}+\varepsilon_{m-3}w^{m-2}+\cdots
+\varepsilon_1w^2+w-\varepsilon_{m-1}.
\end{eqnarray*}
Since the number of $-1$ in $(\varepsilon_{m-3},\varepsilon_{m-4},\ldots,
\varepsilon_1,1,-\varepsilon_{m-1})$ is even, this implies $-\overline{B_j}\in A$.
If $\varepsilon_{m-2}=-1$, then
\begin{eqnarray*}
\overline{B_j} &=& w^{-m+1}+\varepsilon_1w^{-m+2}+\cdots
+\varepsilon_{m-3}w^{-2}-w^{-1}+\varepsilon_{m-1}\\
&=&w^{m-1}-\varepsilon_{m-3}w^{m-2}-\cdots-\varepsilon_1w^2-w+\varepsilon_{m-1}.
\end{eqnarray*}
Since the number of $-1$ in $(-\varepsilon_{m-3},\ldots,-\varepsilon_1,-1,\varepsilon_{m-1})$
is even, $\overline{B_j}\in A$.
The above fact make a one-to-one corresponding from $B$ to $A$.
Also we can simplify the summation in Eq.\,(\ref{eq.10}) if we write
$A=\{A_1,A_2,\ldots,A_{2^{m-2}}\}$ and $B=\{B_1,B_2,\ldots,B_{2^{m-2}}\}$:
\begin{eqnarray*}
\lefteqn{\sum_{\varepsilon_j=\pm 1\atop{1\leq j\leq m-1}}
\sign(\varepsilon)\cdot (w^{m-1}+\varepsilon_1w^{m-2}+\cdots+\varepsilon_{m-1})^{2mn+m}}\\
&=& \sum^{2^{m-2}}_{j=1}\left(A_j^{2mn+m}-B_j^{2mn+m}\right) 
\ =\ \sum^{2^{m-2}}_{j=1}\left(A_j^{2mn+m}-\overline{A_j}^{2mn+m}\right) \\
&=& 2\,i\,\Im\left(\sum_{{\footnotesize\sign}(\varepsilon)=1\atop{\varepsilon_j=\pm 1,1\leq j\leq m-1}}
(w^{m-1}+\varepsilon_1w^{m-2}+\cdots+\varepsilon_{m-1})^{2mn+m}\right).
\end{eqnarray*}
Now $w^{-m(m-1)/2}\cdot i=(-1)^{1+\frac m2}$, this completes our proof.
\end{proof}
Here are some explicit evaluations:
\begin{eqnarray}\label{eq.11} 
\zeta(\{4\}^n) &=& \frac{2^{2n+1}\pi^{4n}}{(4n+2)!},\\ \label{eq.12} 
\zeta(\{8\}^n) &=& \frac{2^{4n+1}\pi^{8n}}{(8n+4)!}\left\{
\left(2\cos\frac{\pi}8\right)^{8n+4}+\left(2\sin\frac\pi 8\right)^{8n+4}\right\},\\
\zeta(\{12\}^n)&=&\frac{3\cdot 2^{6n-1}\pi^{12n}}{(12n+6)!} \label{eq.13} 
\left\{2^{12n+6}+(\sqrt{3}+1)^{12n+6}+(\sqrt{3}-1)^{12n+6}\right\}.
\end{eqnarray}
%
%
\section{Evaluations of $\zeta^\star(\{2m\}^n)$} \label{sec.04}
There are several evaluations of $\zeta^\star(\{2m\}^n)$ available such as 
Hoffman \cite{Hof} or S. Muneta \cite[Theorem A]{Mun}, the later is in terms of Bernoulli numbers as
\begin{eqnarray*}
\lefteqn{\zeta^\star(\{2m\}^n)=\pi^{2mn}\times}\\
&&\left\{\sum_{k_0+\cdots+k_{m-1}=mn\atop k_i\geq 0}(-1)^{m(n-1)}\left(
\prod^{m-1}_{j=0}\frac{(2^{2k_j}-2)B_{2k_j}}{(2k_j)!}\right)
\exp\left(\frac{2\pi i}m\sum^{m-1}_{\ell =0}\ell\cdot k_\ell\right)\right\},
\end{eqnarray*}
for all positive integers $m$ and $n$. However, we note that there is a little difference when $m$ is even or odd. Here we prove a slight
different version in Bernoulli polynomials.
\begin{theorem}\label{thm.03} 
Suppose that $m$, $n$ are positive integers with $m$ even. Then
$$
\zeta^\star(\{2m\}^n)=(2\pi)^{2mn}\sum_{|\bm p|=mn}\left(\prod^m_{j=1}
\frac{B_{2p_j}\left(\frac12\right)}{(2p_j)!}\right)
\exp\left(\frac{2\pi i}{m}\sum^m_{\ell=1}(\ell-1)p_\ell\right).
$$
\end{theorem}
\begin{proof}
Let $w=\exp(2\pi i/2m)$ be the $2m$-th root of unity. 
The generating function of $\zeta^\star(\{2m\}^n)$ is 
\begin{eqnarray*}
1+\sum^\infty_{n=1}\zeta^\star(\{2m\}^n)x^{2mn}
&=&\prod^\infty_{n=1}\left(1-\frac{x^{2m}}{n^{2m}}\right)^{-1} \\
&=& \prod^\infty_{n=1}\prod^m_{j=1}\left(1-\frac{(w^{j-1}x)^2}{n^2}\right)^{-1}\\
&=& \prod^m_{j=1}\frac{2\pi i(w^{j-1}x)e^{\pi i(w^{j-1}x)}}{e^{2\pi i(w^{j-1}x)}-1}.
\end{eqnarray*}
The coefficient of $x^{2mn}$ in the power series expansion of the above product is 
$$
\sum_{|\bm p|=mn}\prod^m_{j=1}\frac{(2\pi i w^{j-1})^{2p_j}}{(2p_j)!}
B_{2p_j}\left(\frac12\right).
$$
Then we write it as  
$$
(2\pi)^{2mn}\sum_{|\bm p|=mn}\left(\prod^m_{j=1}
\frac{B_{2p_j}\left(\frac12\right)}{(2p_j)!}\right)
\exp\left(\frac{2\pi i}{m}\sum^m_{\ell=1}(\ell-1)p_\ell\right).
$$
\end{proof}
\begin{theorem}\label{thm.04} 
Suppose that $m$, $n$ are positive integers with $m$ odd and $m\geq 3$. Then
$$
\zeta^\star(\{2m\}^n)=
(2\pi i)^{2mn}\sum_{|\bm p|=2mn}\left(\prod^m_{j=1}
\frac{B_{p_j}}{(p_j)!}\right)
\exp\left(\frac{2\pi i}{m}\sum^m_{\ell=1}(\ell-1)p_\ell\right).
$$
\end{theorem}
\begin{proof}
Let $w=\exp(2\pi i/m)$ be the $m$-th root of unity. 
The generating function of $\zeta^\star(\{2m\}^n)$ is 
\begin{eqnarray*}
1+\sum^\infty_{n=1}\zeta^\star(\{2m\}^n)x^{2mn}
&=&\prod^\infty_{n=1}\left(1-\frac{x^{2m}}{n^{2m}}\right)^{-1} \\
&=& \prod^\infty_{n=1}\prod^m_{j=1}\left(1-\frac{(w^{j-1}x)^2}{n^2}\right)^{-1}\\
&=& \prod^m_{j=1}\frac{2\pi i(w^{j-1}x)e^{\pi i(w^{j-1}x)}}{e^{2\pi i(w^{j-1}x)}-1}.
\end{eqnarray*}
As $m\geq 3$,
$$
1+w+w^2+\cdots+w^{m-1}=\frac{1-w^m}{1-w}=0,
$$
so the product is equal to 
$$
\prod^m_{j=1}\frac{2\pi i(w^{j-1}x)}{e^{2\pi i(w^{j-1}x)}-1}.
$$
The coefficient of $x^{2mn}$ in the power series expansion at $x=0$ of the above product is 
$$
\sum_{|\bm p|=2mn}\prod^m_{j=1}\frac{(2\pi i w^{j-1})^{p_j}}{(p_j)!}B_{p_j}.
$$
It is also written as 
$$
(2\pi i)^{2mn}\sum_{|\bm p|=2mn}\left(\prod^m_{j=1}
\frac{B_{p_j}}{(p_j)!}\right)
\exp\left(\frac{2\pi i}{m}\sum^m_{\ell=1}(\ell-1)p_\ell\right).
$$
This proves our theorem.
\end{proof}

\begin{remark}
During the proof of the above theorem, if we stop at the step 
$$
\prod^\infty_{n=1}\left(1-\frac{x^{2m}}{n^{2m}}\right)^{-1} 
= \prod^m_{j=1}\frac{2\pi i(w^{j-1}x)e^{\pi i(w^{j-1}x)}}{e^{2\pi i(w^{j-1}x)}-1}
$$
and then expand into power series, we obtain another expression of 
$\zeta^{\star}(\{2m\}^n)$ with $m$ odd:
\begin{eqnarray}\label{eq.14} 
\zeta^\star(\{2m\}^n)=(\pi i)^{2mn}\sum_{|\bm q|=mn}\left(\prod^m_{j=1}
\frac{2^{2q_j}}{(2q_j)!}B_{2q_j}\left(\frac 12\right)\right)
\exp\left(\frac{4\pi i}m\sum^m_{\ell=1}(\ell-1)q_\ell\right).
\end{eqnarray}
\end{remark}
Here are some explicit evaluations:
\begin{eqnarray} \label{eq.15} 
\zeta^\star(\{4\}^n) &=& (2\pi)^{4n}\sum_{|\bm q|=2n}(-1)^{q_2}\frac{B_{2q_1}\left(\frac 12\right)
B_{2q_2}\left(\frac 12\right)}{(2q_1)!(2q_2)!},\\ \label{eq.16} 
\zeta^\star(\{6\}^n) &=& (-1)^n(2\pi)^{6n}\sum_{|\bm q|=6n}\frac{B_{q_1}B_{q_2}B_{q_3}}
{q_1!q_2!q_3!}\exp\left(\frac{2\pi i}{3}(q_2+2q_3)\right),\\
\zeta^\star(\{8\}^n) &=& (2\pi)^{8n}\sum_{|\bm q|=4n}\left(\prod^4_{j=1}
\frac{B_{2q_j}\left(\frac 12\right)}{(2p_j!)}\right) \label{eq.17} 
\exp\left(\frac{\pi i}2\sum^4_{\ell=1}(\ell-1)q_\ell\right).
\end{eqnarray}

The evaluation of $E(4n,k)$ was conjectured in \cite{Gen} as 
$$
(-1)^k\frac{(-4\pi^4)^n}{(4n+2)!}\sum^{n-k}_{u=0}{n-u\choose k}
{4n+2\choose 4u}Y_u,
$$
where
$$
Y_u=\frac 2{(-4)^u}\sum^{2u}_{u_1=0}(-1)^{u_1}{4u\choose 2u_1}
(2-4^{u_1})B_{2u_1}(2-4^{2u-u_1})B_{2(2u-u_1)}.
$$
Now it becomes our Theorem C in a slight different notation.
\begin{thmC} 
For a pair of positive integers $n$, $k$ with $1\leq k\leq n$, we have
\begin{eqnarray*}
\lefteqn{E(4n,k)} \\
&=& \frac{2^{2n+1}\pi^{4n}}{(4n+2)!}\sum^{n-k}_{q=0}(-1)^{n-q-k}
{n-q\choose k}{4n+2\choose 4q}\sum_{|\bm q|=2q}
{4q\choose 2q_1}(-1)^{q_1}2^{2q}B_{2q_1}\left(\frac 12\right)B_{2q_2}\left(\frac 12\right).
\end{eqnarray*}
\end{thmC}
\begin{proof}
Substituting $m=4$ in the formula of $E(mn,k)$ in Theorem A, we have
$$
E(4n,k)=\sum_{p+q=n}(-1)^{p-k}{p\choose k}\zeta(\{4\}^p)\zeta^\star(\{4\}^q).
$$
By Eq.\,(\ref{eq.11}) and Eq.\,(\ref{eq.15}), $E(4n,k)$ becomes
$$
\sum_{p+q=n}(-1)^{p-k}{p\choose k}
\frac{2^{2p+1}\pi^{4p}}{(4p+2)!}
\sum_{|\bm q|=2q}\frac{(-1)^{q_1}(2 \pi)^{4q}}{(2q_1)!(2q_2)!}
B_{2q_1}\left(\frac 12\right) B_{2q_2}\left(\frac 12\right)
$$
Note that the summation is equal to zero if $p<k$, so we can rewrite it as the form stated in the theorem.
\end{proof}

Based on the evaluations of $\zeta(\{6\}^p)$ and $\zeta^\star(\{6\}^q)$ (see Eq.\,(\ref{eq.06})
and Eq.\,(\ref{eq.16})), we have the following evalution of $E(6n,k)$:
\begin{eqnarray*}
E(6n,k) &=& \frac{(-1)^{n-k}6(2\pi)^{6n}}{(6n+3)!}
\sum^{n-k}_{q=0}{n-q\choose k}{6n+3\choose 6q}\\
&&\qquad\qquad\times 
\sum_{|\bm q|=6q}\frac{\exp(2\pi i(q_2+2q_3)/3)(6q)!}{(q_1)!(q_2)!(q_3)!}
B_{q_1} B_{q_2}B_{q_3},
\end{eqnarray*}
or using Eq.(\ref{eq.14}) we have another expression 
\begin{eqnarray*}
E(6n,k) &=& \sum_{|\bm\alpha|=n}\zeta(6\alpha_1,6\alpha_2,\ldots,6\alpha_k) \\
&=& \frac{(-1)^{n-k}6(2\pi)^{6n}}{(6n+3)!}
\sum^{n-k}_{q=0}{n-q\choose k}{6n+3\choose 6q}\\
&&\qquad\qquad\times 
\sum_{|\bm q|=3q}\frac{\exp(4\pi i(q_2+2q_3)/3)(6q)!}{(2q_1)!(2q_2)!(2q_3)!}
B_{2q_1}\left(\frac 12\right) B_{2q_2}\left(\frac 12\right) B_{2q_3}\left(\frac 12\right).
\end{eqnarray*}

At last we give the evaluation of $E(8n,k)$:
\begin{eqnarray*}
\lefteqn{E(8n,k)}\\
&=& \sum_{|\bm\alpha|=n}\zeta(8\alpha_1,8\alpha_2,\ldots,8\alpha_k) \\
&=& \frac{2^{4n+1}\pi^{8n}}{(8n+4)!}\sum^{n-k}_{q=0}(-1)^{n-q-k}{n-q\choose k}
{8n+4\choose 8q}
\left\{\left(2\cos\frac{\pi}8\right)^{8n-8q+4}
+\left(2\sin\frac{\pi}8\right)^{8n-8q+4}\right\}\\
&&\qquad\qquad\times
\sum_{|\bm q|=4q}2^{4q}(8q)!\left(\prod^4_{j=1}
\frac{B_{2q_j}\left(\frac12\right)}{(2q_j)!}\right)
\exp\left(\frac{\pi i}2\sum^4_{\ell=1}(\ell-1)q_\ell\right).
\end{eqnarray*}

%
%
\section{A final remark}\label{sec.05}
According to our Theorem A, the evaluation of $E(mn,k)$ depends on $\zeta(\{m\}^p)$ and $\zeta^\star(\{m\}^q)$ with $p+q=n$. Both multiple zeta values can be evaluated
in terms of $\zeta(m), \zeta(2m),\ldots, \zeta(nm)$.

Here we list some preliminaries about symmetric functions. For more details, we refer the reader to \cite{Mac, Sta}. 
Let $e_m$, $h_m$, and $p_m$ be the $m$-th elementary, complete homogeneous,
and power-sum symmetric polynomials, 
in infinitely many variables $x_1, x_2,\ldots$, respectively.
They have associated generating functions
$$\begin{array}{lllllllll}
E(t)&:=&\displaystyle\sum^\infty_{j=0}e_jt^j 
     &=&\displaystyle\prod^\infty_{i=1}(1+tx_i),\\
H(t)&:=&\displaystyle\sum^\infty_{j=0}h_jt^j 
      &=&\displaystyle\prod^\infty_{i=1}\frac 1{1-tx_i} &=& E(-t)^{-1},\\
P(t)&:=&\displaystyle\sum^\infty_{j=1}p_jt^{j-1} 
      &=&\displaystyle\sum^\infty_{i=1}\frac{x_i}{1-tx_i} 
      &=&\displaystyle\frac{H'(t)}{H(t)}
      &=&\displaystyle\frac{E'(-t)}{E(-t)}.
\end{array}$$

Let the modified Bell polynomials $P_m(x_1,x_2,\ldots,x_m)$ be 
defined by \cite{Chen,CC}
$$
\exp\left(\sum^\infty_{k=1}\frac{x_k}{k}z^k\right)
=\sum^\infty_{m=0}P_m(x_1,x_2,\ldots,x_m)z^m.
$$
The general explicit expression for $P_m$ is
$$
P_m(x_1,\ldots,x_m)=
\sum_{k_1+2k_2+\cdots+mk_m=m}
\frac 1{k_1!\cdots k_m!}
\left(\frac{x_1}1\right)^{k_1}\cdots
\left(\frac{x_m}m\right)^{k_m}.
$$
\begin{lemma} \label{lma.01} 
Let $j$ be a nonnegative integer. Then we have
\begin{eqnarray} 
e_j &=& P_j(p_1,-p_2,\ldots,(-1)^{j+1}p_j),\label{eq.18}\\  
h_j &=& P_j(p_1,p_2,\ldots,p_j). \label{eq.19} 
\end{eqnarray}
\end{lemma}
\begin{proof}
We begin with the generating function of $e_j$.
\begin{eqnarray*}
\sum^\infty_{j=0}e_jt^j &=& \prod^\infty_{i=1}(1+tx_i) \ =\  
\exp\left[\sum^\infty_{i=1}\log(1+tx_i)\right] \\
&=& \exp\left[\sum^\infty_{\ell=1}(-1)^{\ell+1}\frac{t^\ell}{\ell}\sum^\infty_{i=1}x_i^\ell\right]\\
&=& \sum^\infty_{m=0}P_m(p_1,-p_2,\ldots,(-1)^{m+1}p_m)t^m.
\end{eqnarray*}
This gives us Eq.\,(\ref{eq.18}). Eq.\,(\ref{eq.19}) can be easily derived by a similar method,
hence we omit it.
\end{proof}
Let $x_k=k^{-s}$, for all $k\geq 1$. 
Then
$$
e_n=\zeta(\{s\}^n),
\quad h_n=\zeta^\star(\{s\}^n),
\quad\mbox{and}\quad p_n=\zeta(sn).
$$
Note that we denote that $\zeta(\{s\}^0)=\zeta^*(\{s\}^0)=1$. Now Eq.\,(\ref{eq.18}) and 
Eq.\,(\ref{eq.19}) give us the evaluations of $\zeta(\{s\}^n)$ and $\zeta^\star(\{s\}^n)$:
\begin{eqnarray}\label{eq.20} 
\zeta(\{s\}^n) &=& P_n(\zeta(s),-\zeta(2s),\ldots,(-1)^{n+1}\zeta(ns)), \\
\zeta^\star(\{s\}^n) &=& P_n(\zeta(s),\zeta(2s),\ldots,\zeta(ns)). \label{eq.21} 
\end{eqnarray}
For example, we have
\begin{eqnarray*}
\zeta(\{s\}^5) &=& \frac 1{5!}\left(\zeta(s)^5-10\zeta(s)^3\zeta(2s)+20\zeta(s)^2\zeta(3s)
                                  -30\zeta(s)\zeta(4s)\right.\\
&&\qquad\qquad\left.+15\zeta(s)\zeta(2s)^2-20\zeta(2s)\zeta(3s)+24\zeta(5s)\right);\\
\zeta^\star(\{s\}^5) &=& \frac 1{5!}\left(\zeta(s)^5+10\zeta(s)^3\zeta(2s)+20\zeta(s)^2\zeta(3s)
                                  +30\zeta(s)\zeta(4s)\right.\\
&&\qquad\qquad\left.+15\zeta(s)\zeta(2s)^2+20\zeta(2s)\zeta(3s)+24\zeta(5s)\right).
\end{eqnarray*}
We can give another evaluations of $\zeta(\{2m\}^n)$ and 
$\zeta^\star(\{2m\}^n)$ in terms of Bernoulli numbers.
\begin{proposition} \label{prop.03} 
Let $m$, $n$ be positive integers. Then we have
$$
\zeta(\{2m\}^n)=(-1)^{n(m+1)}(2\pi)^{2mn}
\sum_{a_1+2a_2+\cdots+na_n=n\atop a_i\geq 0}
\prod^n_{k=1}\frac 1{a_k!}\left(\frac{B_{2km}}{2k\cdot (2km)!}\right)^{a_k} 
$$
and 
$$
\zeta^\star(\{2m\}^n)=(-1)^{mn}\cdot(2\pi)^{2mn}
\sum_{a_1+2a_2+\cdots+na_n=n\atop a_i\geq 0}
\prod^m_{k=1}\frac{(-1)^{a_k}}{a_k!}\left(\frac{B_{2km}}{2k\cdot (2km)!}\right)^{a_k}. 
$$
\end{proposition}
Note that Eq.\,(\ref{eq.20}) and Eq.\,(\ref{eq.21}) could be obtained by setting
$\mathbf i=(s,s,\ldots,s)$ in Eq.\,(\ref{eq.01}) and Eq.\,(\ref{eq.02}), respectively.

Now we can refined Theorem A as follows.
\begin{theorem}\label{thm.05} 
For a pair of positive integers $m$, $n$, with $m\geq 2$, we have 
for $1\leq k\leq n$ that 
\begin{eqnarray*}
E(mn,k)&=&\sum_{p+q=n}(-1)^{p-k}{p\choose k} \\
&& \hspace{12mm}\times P_p(\zeta(m),-\zeta(2m),\ldots,(-1)^{p+1}\zeta(pm))P_q(\zeta(m),\zeta(2m),\ldots,\zeta(qm)).
\end{eqnarray*}
\end{theorem}

%
%


%
%
%


\begin{thebibliography}{99}

\bibitem{AK}
T. Arakawa, M. Kaneko,
\textsl{A primer of multiple zeta values},
lecture notes (in Japanese), 2010.

\bibitem{BBBL1}
J. M. Borwein, D. M. Bradley, D. J. Broadhurst, P. Lisonek,
Combinatorial aspects of multiple zeta values,
\textsl{Electron. J. Combin.},
5 (1998), Research paper 38, 12 pp. (electronic).

\bibitem{BBBL2}
J. M. Borwein, D. M. Bradley, D. J. Broadhurst, P. Lisonek,
Special values of multiple polylogarithm,
\textsl{Trans. Amer. Math. Soc.},
353 (2001), no. 3, 907--941.

\bibitem{Chen}
Kwang-Wu Chen,
Generalized Harmonic Numbers and Euler Sums,
\textsl{Int. J. Number Theory},
Online Ready (2015), pp 1--18. DOI:10.1142/S1793042116500883.

\bibitem{CC}
M.-A. Coppo, B. Candelpergher,
The Arakawa-Kaneko zeta function,
\textsl{Ramanujan J.},
22.2 (2010), 153--162. 


\bibitem{Eie}
M.~Eie, 
\emph{Topics in Number Theory}, 
Monographs in Number Theory vol.2, 
World Scientific Publishing Co. Pte. Ltd, 2009.

\bibitem{Eie2} 
M.~Eie, 
\emph{The Theory of Multiple Zeta Values with Applications in Combinatorics}, 
Monographs in Number Theory vol.7, 
World Scientific Publishing Co. Pte. Ltd, 2013.

\bibitem{GKZ}
H.~Gangl, M.~Kaneko, D.~Zagier,
\emph{Double zeta values and modular forms},
in Automorphic Forms and Zeta Functions.
In Memory of Tsuneo Arakawa.
Proceedings of the Conference,
Rikkyo University, Tokyo, Japan, September 2004, 71--106,
World Scientific, Hackensack, NJ, USA, 2006.

\bibitem{Gen}
M.~Gen\v{c}ev,
On restricted sum formulas for multiple zeta values with even arguments,
\textsl{Arch. Math.},
107 (2016), 9--22.

\bibitem{Hof}
M.~E.~Hoffman,
Multiple Harmonic Series,
\textsl{Pac. J. of Math.},
152 (1992), no. 2, 275--290.

\bibitem{Hof2}
M.~E.~Hoffman,
On multiple zeta values of even arguments,
\textsl{arXiv: 1205.7051v4} (2016).

\bibitem{Mac}
I. G. MacDonald,
\textsl{Symmetric Functions and Hall Polynomials},
$2$nd edition, Claredon Press, 1995.

\bibitem{Mun}
S. Muneta,
On some explicit evaluations of multiple zeta-star values,
\textsl{J. of Number Theory},
128 (2008), 2538--2548.

\bibitem{SC}
Z.~Shen, T.~Cai,
Some identities for multiple zeta values,
\textsl{J. of Number Theory},
132 (2012), no. 2, 314--323.

\bibitem{Sta}
R. P. Stanley,
\textsl{Enumerative Combinatorics},
vol. $2$, Cambridge University Press, 1999.

\bibitem{Rad}
H.~Rademacher,
\emph{Topics in Analytic Number Theory},
Grundlehren der mathematischen Wissenschaften 169, Springer-Verlag Berlin Heidelberg, 1973.

\bibitem{Zag}
D. Zagier, 
Evaluation of the multiple zeta value $\zeta(2,\ldots,2,3,2,\ldots,2)$,
\textsl{Annals of Mathematics},
175 (2012), 977-1000.

\end{thebibliography}
\end{document}